\documentclass[12pt, reqno]{amsart}
\usepackage{amsmath, amsthm, amscd, amsfonts, amssymb, graphicx, color}

\newtheorem{df}{Definition}[section]
\newtheorem{thm}[df]{Theorem}
\newtheorem{pro}[df]{Proposition}

\begin{document}
\setcounter{page}{1}

\title[Quasi-Solvable Lie Algbras]{Joint Spectrum for quasi-solvable\\ Lie algebras of operators}
\author{Enrico Boasso}

\begin{abstract}Given a complex Banach space $X$ and a joint spectrum for complex solvable finite dimensional Lie algebras of operators defined on $X$, we extend this joint
spectrum to quasi-solvable Lie algebras of operators, and we prove the main spectral properties of the extended
joint spectrum. We also show that this construction is uniquely determined by the original joint spectrum.\end{abstract}

\maketitle
\section{ Introduction}

\indent Given a Banach space $X$, Z. S\l odkowski and W. Zelazko 
studied in [9] the main spectral properties of joint spectra 
of commuting, either finite or infinite, families of operators defined on $X$. 
In addition, in order to prove the projection property for these joint 
spectra, they showed that if a joint spectrum is defined only for 
finite families of mutually commuting operators and  it has the 
projection property, then, by means of the notion of inverse limit, 
this joint spectrum can be uniquely extended to a joint spectrum 
defined on the family of all subsets of ${\mathcal L} (X)$ consisting 
of pairwise commuting operators. Moreover, 
this generalized joint spectrum also has the main spectral 
properties, i.e., it is a compact nonempty set and the 
projection property still holds.\par

\indent On the other hand, in the last years some joint spectra
for operators generating nilpotent and solvable Lie algebras
were introduced. For example, in [5] it was considered the first
non-commutative version of the Taylor joint spectrum, [10],
for nilpotent Lie algebras of operators. Working independently,
in [2] and in [3] we extended the Taylor, and the 
S\l odkowki joint spectra, [10] and [8], of finite 
commuting tuples of operators to complex 
solvable finite dimensional Lie algebras of 
operators. Moreover, in [1] it was introduced a new concept of 
spectrum for complex solvable finite dimensional Lie algebras 
of operators, which coincides with one 
of [5] and [2] in the case of a nilpotent Lie algebra, but in the 
solvable case differs, in general,  from the one of [2]. In addition, 
the joint spectrum introduced in [1] may be extended to some 
infinite dimensional Lie algebras of operators, in a similar way 
as Z. S\l odkowski and W. Zelazko did in [9] for the commutative case.
\par    

\indent In this article we consider another non commutative variant of 
the construction developed by Z. S\l odkowski and W. Zelazko in [9]. 
In fact, we consider a joint spectrum defined for complex solvable finite 
dimensional Lie algebras of operators, such as the Taylor and the S\l odkowski 
joint spectra, and we extend it to quasi-solvable Lie algebra of operators, 
see section 2. In addition, we prove the main spectral properties for 
our joint spectrum, and we show that this construction is 
uniquely determined by the original joint spectrum. On the other hand, 
this joint spectrum will be, in general, different from the one considered in [1].\par   

\indent The paper is organized as follows. In Section 2 we recall several 
definitions and results which we need for our work, and in Section 3 we 
prove our main result.\par
    
\section{Preliminaries}

\indent Let us begin with the definition of joint spectrum 
which we shall consider. \par 

\begin{df} Let $X$ be a Banach space. A joint spectrum is a 
function, $\sigma$, which assigns, to each complex solvable finite 
dimensional Lie algebra $L$ of operators defined on $X$, 
a compact nonempty subset $\sigma (L)$ of characters, such that if 
$H$ is a Lie ideal of $L$, and $\pi$ is the restriction map $\pi\colon 
L^*\to H^*$, then the projection property for ideals holds, i.e., $
\pi(\sigma (L))=\sigma (H)$.
\end{df} 

\indent The joint spectra that we are considering are the Taylor joint 
spectrum and the S\l odkowski joint spectra for complex
solvable finite dimensional Lie algebras of operators, see
[8], [10], [2], [3] and [6].\par

\indent We now recall the definition of a quasi-solvable Lie algebra.\par

\begin{df} A quasi-solvable Lie algebra ${\mathcal L}$ is a complex Lie algebra such that ${\mathcal L}=\sum_{\alpha\in I} I_{\alpha}$, where $I$ is an index set, and for each $\alpha\in I$, $I_{\alpha}$ is a complex solvable finite dimensional ideal of ${\mathcal L}$.\end{df} 

\indent In order to see the main poperties of the quasi-solvable Lie algebras and their behaviour under representations in Banach spaces, 
we refer to [7] and [11].\par

\indent Our main result concers the notion of inverse limit. We recall the most important facts related to this notion; for a complete exposition see [4].\par

\begin{df} An inverse system of sets and maps $\{{\mathcal X},\hbox{}\pi\}$, over a directed set $(M,<)$, is a function which attaches to each $\alpha\in M$ a set ${\mathcal X}_{\alpha}$, and to each pair $\alpha$ and $\beta$ such that $\alpha<\beta$ in $M$, a map $\pi_{\alpha}^{\beta}\colon {\mathcal X}_{\beta}\to {\mathcal X}_{\alpha}$ such that
$$
\pi_{\alpha}^{\alpha}=Id_M,\hskip2cm \pi_{\alpha}^{\beta}\circ\pi_{\beta}^{\gamma}=\pi_{\alpha}^{\gamma},
$$
where $\alpha$, $\beta$, $\gamma$ belong to $M$ and $\alpha <\beta<\gamma$.\end{df}  

\begin{df} Let $\{{\mathcal X},\pi\}$ be an inverse system of sets and maps over a directed set $(M,<)$. Then, the inverse limit ${\mathcal X}_{\infty}$ is the subset
of the product $\Pi_{\alpha\in M} X_{\alpha}$ consisting of those elements $x=(x_{\alpha})_{\alpha\in M}$, such that for each relation $\alpha<\beta$ in $M$,  $\pi_{\alpha}^{\beta}(x_{\beta})=x_{\alpha}$.
\end{df}
\indent If all the sets ${\mathcal X}_{\alpha}$ are topological space, 
then to ${\mathcal X}_{\infty}$ is assigned the topology as subspace of
$\Pi_{\alpha\in M} {\mathcal X}_{\alpha}$. Naturally, the projections 
$$
\pi_{\alpha}\colon {\mathcal X}_{\infty}\to {\mathcal X}_{\alpha},\hskip2cm \pi_{\alpha}(x)=x_{\alpha},
$$ 
are continuous maps. Moreover, if for each $\alpha\in I$, ${\mathcal X}_{\alpha}$ is a nonempty compact space, then ${\mathcal X}_{\infty}$ is also nonempty and compact, see [4, Chapter VIII, Section 3, Lemma 3.3,Theorem 3.6].\par

\section{The Main Result}

\indent In this section we extend the joint spectrum
which we have defined from complex solvable finite dimensional  Lie algebras of operators to quasi-solvable Lie algebras of operators. As we have said, in order to define this joint spectrum and to prove its main spectral properties, we
work with the notion of inverse limit as in [9] and in [1].
On the other hand, we first give a definition of the joint spectrum which depends on a particular presentation of the algebra ${\mathcal L}$, and then we show that this definition is the correct one. We proceed as follows.\par

\indent Let us consider $X$ a Banach space, ${\mathcal L}$ a complex 
quasi-solvable Lie subalgebra of ${\mathcal L}(X)$ and an index set $I$ 
such that ${\mathcal L}=\sum_{\alpha\in I} I_{\alpha}$, where, for 
each $\alpha\in I$, $I_{\alpha}$ is a complex solvable finite 
dimensional ideal of ${\mathcal L}$, and such that $I$, with the 
inclusion, is a directed set, i.e., if $I_{\alpha}$ and $I_{\beta}$ 
are solvable ideals of ${\mathcal L}$ such that $\alpha$ and $\beta$ 
belong to $I$, then there is a solvable ideal $I_{\gamma}$, 
$\gamma\in I$, such that $I_{\alpha}\cup I_{\beta}
\subseteq I_{\gamma}$. For example,
if $\mathbb I$ is the set of all complex solvable finite dimensional ideals of ${\mathcal L}$,
then $\mathbb I$ is a directed set. \par

\indent In addition, given $I$ is as above, let us consider the family of maps $\pi=\{\pi_{\alpha}^{\beta}\colon I_{\beta}^*\to I_{\alpha}^*\}$, where $\alpha$ and $\beta$ belong to $I$,
$\alpha <\beta$, and $\pi$  is the usual pojection, i.e., the restriction map.
Let us also consider the family of sets ${\mathcal X}=\{ \sigma(I_{\alpha})\}_{\alpha\in I}$. Then, by the projection property of the joint spectrum, we have that $\{{\mathcal X},\pi\}$ is an inverse system of topological spaces. Now, we may state our definition of
the joint spectrum for quasi-solvable Lie algebras.\par

\begin{df} Let $X$, ${\mathcal L}$, $I$, $( I_{\alpha})_{\alpha\in I}$ and $\{{\mathcal X},\pi\}$ be as above. The joint spectrum of the quasi-solvable algebra ${\mathcal L}$, relative to the presentation of ${\mathcal L}$ defined by $I$ and by $(I_{\alpha})_{\alpha\in I}$, is the inverse limit of the inverse system $\{{\mathcal X},\pi\}$, and it is denoted by 
$$
\sigma(\mathcal L, (I_{\alpha})_{\alpha\in I})={\mathcal X}_{\infty}.
$$
\end{df}

\indent We observe that this definition depends on the set $I$ and on a particular presentation of ${\mathcal L}$, however, Proposition 3.4 shows that the extended joint spectrum is independent of  the presentation of $\mathcal L$, and Theorem 3.7 that it is uniquely determined by its properties.\par

\indent On the other hand, by [4, Chapter VIII, Section 3, Theorem 3.6] we have that $\sigma(\mathcal L, (I_{\alpha})_{\alpha\in I})$ is a compact nonempty subset of
$\Pi_{\alpha\in I} \sigma(I_{\alpha})$.  Let us now study in more detail the properties of the introduced joint spectrum.\par

\begin{pro} Let $X$, $\mathcal L$, $I$, 
$(I_{\alpha})_{\alpha\in I}$ and $\{{\mathcal X},\pi\}$ 
be as above. Then, the joint spectrum 
$\sigma(\mathcal L,(I_{\alpha})_{\alpha\in I})$ 
may be identified with a subset of the characters of $\mathcal L$.\end{pro}

\begin{proof}

\indent Let $(f_{\alpha})_{\alpha\in I}$ belongs to $\sigma ( L, (I_{\alpha})_{\alpha\in I})$, and let us associate to this element the function $f$ defined by $f\mid I_{\alpha}=f_{\alpha}$. Let us see that
$f$ is a well defined character of $\mathcal L$.\par

\indent First of all let us consider $x\in I_{\alpha}\cap I_{\beta}$, where $\alpha$ and $\beta\in I$. Thus, since $I$ is a directed set, there is an
ideal $I_{\gamma}$, $\gamma\in I$, such that $I_{\alpha}\cup I_{\beta}\subseteq I_{\gamma}$. Then, since $\{{\mathcal X},\pi\}$ is an inverse system we have
that $f_{\alpha}(x)=f_{\gamma}(x)=f_{\beta}(x)$.\par 

\indent Now if $x\in {\mathcal L}$, let us present it as
$x=\sum_{j=1}^{j=n} x_{\alpha_j}=\sum_{k=1}^{k=m} x_{\alpha_k^{'}}$, where $ x_{\alpha_j}\in I_{\alpha_j}$, 
$ x_{\alpha_k^{'}}\in I_{\alpha_k^{'}}$, and $\alpha_j$ and $\alpha_k^{'}$ belong to $I$,
$1\le j\le n$, $1\le k\le m$. Since $I$ is a directed set, there is a finite dimensional solvable ideal of $\mathcal L$, $I_{\beta}$, $\beta\in I$, such that 
$ \cup_{j=1}^{j=n} I_{\alpha_j}\cup\cup_{k=1}^{k=m} I_{\alpha_k^{'}}\subseteq I_{\beta}$. Then
$$
\sum_{j=1}^{j=n}f_{\alpha_j}(x_{\alpha_j})=f_{\beta}(x)=\sum_{k=1}^{k=m}f_{\alpha_k{'}}(x_{\alpha_k^{'}} ).
$$
Thus, $f$ is a well defined map.\par

\indent By a similar argument it is easy to see that $f$ is a character of $\mathcal L$. Moreover, by construction, the above assigment is an injective identification.
\end{proof}

\begin{pro}Let $X$, $\mathcal L$, $I$, $(I_{\alpha})_{\alpha\in I}$ and $\{{\mathcal X},\pi\}$ be as above. Then
$$
\sigma (\mathcal L, (I_{\alpha})_{\alpha\in I})=\{f\colon  f\hbox{ is a character of }{\mathcal L} \hbox{, and for each } \alpha\in I\hbox{, }f\mid I_{\alpha}\in \sigma (I_{\alpha})\} . 
$$
\end{pro}

\begin{proof}

\indent By the definition of $\sigma(\mathcal L, (I_{\alpha})_{\alpha\in I})$ and by Proposition 3.2, we have that the joint spectrum is contained
in the right hand set of the identity.\par

\indent On the other hand, if $f$ is a character of $\mathcal L$ such that
$f\mid I_{\alpha}\in \sigma (I_{\alpha})$, then by the projection property of the joint spectrum, $(f_{\alpha})_{\alpha\in I}$ belongs to
$\sigma (\mathcal L, (I_{\alpha})_{\alpha\in I})$. However, by Proposition 3.2, $(f_{\alpha})_{\alpha\in I}$  is identified with $f$.
\end{proof} 

\indent As a consequence of the Proposition 3.3 we have that the
joint spectrum is independent of a particular presentation of the
quasi-solvable Lie algebra $\mathcal L$.\par

\begin{pro}Let $X$ and $\mathcal L$ be as above, and $I_j$, $j=1$, $2$, two directed index sets such that
$\mathcal L=\sum_{\alpha_1\in I_1} I_{\alpha_1}= \sum_{\alpha_2\in I_2} I_{\alpha_2}$, where $I_{\alpha_j}$, $j=1$, $2$, are complex solvable finite dimensional ideals
of $\mathcal L$. Then
$$
\sigma(\mathcal L,(I_{\alpha_1})_{\alpha_1\in I_1})=\sigma(\mathcal L,(I_{\alpha_2})_{\alpha_2\in I_2}).
$$
\end{pro}

\begin{proof}

\indent It is enough to see that if $\mathbb I$ the set of all complex solvable finite dimensional ideals of $\mathcal L$,
then $
\sigma(\mathcal L,(I_{\alpha})_{\alpha\in I})=\sigma(\mathcal L,(I_{\alpha})_{\alpha\in \mathbb I})$,
where $I$ is a directed index set and $(I_{\alpha})_{\alpha\in I}$ is a particular presentation of $\mathcal L$.\par

\indent First of all, by Proposition 3.3 it is clear that $\sigma(\mathcal L,(I_{\alpha})_{\alpha\in\mathbb I})\subseteq\sigma(\mathcal L,(I_{\alpha})_{\alpha\in  I})$.\par

\indent On the other hand, given $f\in\sigma(\mathcal L,(I_{\alpha})_{\alpha\in  I})$, let us consider $I_{\alpha_0}$ 
an arbitrary complex solvable finite dimensional ideal of $\mathcal L$. Since $\mathcal L=\sum_{\alpha\in I} I_{\alpha}$, and $I$,
with the inclusion, is a directed set, there is $\beta\in I$ such
that $I_{\alpha_0}\subseteq I_{\beta}$. Now, since 
$f\in\sigma(\mathcal L,(I_{\alpha})_{\alpha\in  I})$, by Proposition 3.3 we have that $f\mid I_{\beta}\in\sigma(I_{\beta}) $. However, by the projection property of the joint spectrum, $\sigma(I_{\beta})\mid I_{\alpha_0}=\sigma (I_{\alpha_0})$.
Then, $f\mid I_{\alpha_0}\in\sigma (I_{\alpha_0})$ and
$f\in\sigma (\mathcal L, (I_{\alpha})_{\alpha\in \mathbb I})$.
\end{proof} 
 
\indent In the following propositions we prove two of the most
important properties of the joint spectrum for quasi-solvable Lie algebras of operators.\par 

\begin{pro}Let $X$, $\mathcal L$, $I$, $(I_{\alpha})_{\alpha\in I}$ and $\{{\mathcal X},\pi\}$ be as above. Then, if $L$ is a complex solvable finite dimensional ideal of $\mathcal L$,
$$
\sigma ( L, (L\cap I_{\alpha})_{\alpha\in I})=\sigma (L).
$$
\end{pro}

\begin{proof}

\indent First of all we observe that $L=\sum_{\alpha\in I} 
L\cap I_{\alpha}$.\par 

\indent Indeed, it is clear that $\sum_{\alpha\in I} 
L\cap I_{\alpha}\subseteq L$. On the other hand, if $x\in L\subseteq
{\mathcal L}$, there are $\alpha_i\in I$ and $x_i\in I_{\alpha_i}$,  
$i=1,\ldots ,n$, such that $x=\sum_{i=1}^n x_i$. However, 
since  $I$ is a directed set with the inclusion, there is $\alpha\in I$ 
such that $\cup_{i=1}^n I_{\alpha_i}\subseteq I_{\alpha}$. Thus,
$x\in L\cap I_{\alpha}\subseteq \sum_{\alpha\in I} 
L\cap I_{\alpha}$.\par

\indent Now, since $L=\sum_{\alpha\in I} 
L\cap I_{\alpha}$, 
we may construct the inverse limit set.\par

\indent In addition, since $L$ is a finite dimensional ideal of $\mathcal L$, and $I$ is a directed set, there is a solvable finite
dimensional ideal of $\mathcal L$, $I_{\beta}$, $\beta\in I$, such that $L=L\cap I_{\beta} $. Now, if $f$ is a character of $L$, by construction of $\sigma ( L, (L\cap I_{\alpha})_{\alpha\in I})$ and of Proposition 3.3 we have that
$f\mid L\cap I_{\beta}\in \sigma (L\cap I_{\beta})$. Since $L=L\cap I_{\beta} $, $f\mid L\cap L_{\beta}=f$, then $\sigma ( L, (L\cap I_{\alpha})_{\alpha\in I})\subseteq \sigma(L)$.\par

\indent On the other hand, if $f\in \sigma (L)$, by the projection property for ideals of the joint spectrum, $f\mid{ L\cap I_{\alpha}}\in \sigma (L\cap I_{\alpha})$, thus, as $f$ is a character of $L$, by Proposition 3.3 we have 
the reverse contention.
\end{proof} 
  
\begin{pro}Let $X$, $\mathcal L$, $I$, 
$(I_{\alpha})_{\alpha\in I}$ and $\{{\mathcal X},\pi\}$ be as above. 
If $\mathcal H$ is an
ideal of $\mathcal L$. Then, the joint spectrum
has the projection property for ideals, i.e.,
$$
\sigma (\mathcal L, (I_{\alpha})_{\alpha\in I})\mid {\mathcal H} =\sigma (\mathcal H,  (\mathcal H\cap I_{\alpha})_{\alpha\in I}).
$$
\end{pro}

\begin{proof}

\indent First of all, let us consider the directed set $I$. Then, an easy calculation shows that $\mathcal H=\sum_{\alpha\in I}H\cap I_{\alpha}$.
Thus, we may consider the set $\sigma (\mathcal H, (\mathcal H\cap I_{\alpha})_{\alpha\in I})$.\par

\indent Now we consider the inverse systems $\{{\mathcal X},\pi\}$ and $\{{\mathcal X}^{'},\pi^{'}\}$ defined by ${\mathcal X}_{\alpha}= \sigma (I_{\alpha})$,
${\mathcal X}_{\alpha}^{'}= \sigma (\mathcal H\cap I_{\alpha})$, and  $\pi$
and $\pi^{'}$ are the families of projection maps defined between the corresponding spectral sets. Then, if we consider the map $Id\colon I\to I$ and  for each $\alpha\in I$ we also consider $P_{\alpha}\colon \sigma (I_{\alpha})\to \sigma (\mathcal H\cap I_{\alpha})$, the canonical restriction map, then, it is easy to verify that the identity map of $I$
and the family $(P_{\alpha})_{\alpha\in I}$ is a map of inverse systems, see [4, Chapter VIII, Section 2, Definition 2.3]. Thus, we have a well defined and continuos map $P_{\infty}\colon \sigma (\mathcal L, (I_{\alpha})_{\alpha\in I})\to \sigma (\mathcal H,  (\mathcal H\cap I_{\alpha})_{\alpha\in I})$, see [4, Chapter VIII, Section 3, Definition 3.10, Lemma 3.11, Theorem 3.13]. However, by the identification of Proposition 3.3 we have that
$$
\sigma (\mathcal L,  (I_{\alpha})_{\alpha\in I})\mid{\mathcal H} =P_{\infty}
(\sigma (\mathcal L,  (I_{\alpha})_{\alpha\in I}))\subseteq
\sigma (\mathcal H,  (\mathcal H\cap I_{\alpha})_{\alpha\in I}).
$$

\indent On the other hand, let us suppose that there is $f\in \sigma (\mathcal H,  (\mathcal H \cap I_{\alpha})_{\alpha\in I})\setminus P_{\infty}
(\sigma (\mathcal L, (I_{\alpha})_{\alpha\in I})$. Since $\sigma (\mathcal L, (I_{\alpha})_{\alpha\in I})$ is a compact set and $P_{\infty}$ is a continous map, there is an open set $U$, which may be chosen in the base of the topology of $\sigma (\mathcal H,  (\mathcal H\cap I_{\alpha})_{\alpha\in I})$, such that $f\in U$ and that $U\cap P_{\infty}
(\sigma (\mathcal L, (I_{\alpha})_{\alpha\in I}))$ is the empty set.\par

\indent In addition, by [4, Chapter VIII, Section 3, Definition 3.1, Lemma 3.3, Corollary 3.9] we have two well defined families of surjective and continuous maps, $(\pi_{\alpha})_{\alpha\in I_{\alpha}}$ and $(\pi_{\alpha}^{'})_{\alpha\in I_{\alpha}}$, which satisfies
$$
\pi_{\alpha}\colon \sigma (\mathcal L, (I_{\alpha})_{\alpha\in I})\to \sigma (I_{\alpha}),
$$
$$
 \pi_{\alpha}^{'}\colon \sigma (\mathcal H,  (\mathcal H\cap I_{\alpha})_{\alpha\in I})\to \sigma (\mathcal H\cap I_{\alpha}).
$$
Moreover, it is easy to see that $\pi^{'}_{\alpha}\circ P_{\infty}=
P_{\alpha}  \circ \pi_{\alpha}$, [4, Chapter VIII, Section 3, Lemma 3.11].\par
 
\indent Now, by [4, Chapter VIII, Section 3, Lemma 3.12] we know that there is an $\alpha\in I$
and $V$ in the topology of $\sigma (\mathcal H\cap I_{\alpha} )$ such that
$U=\pi^{'-1}_{\alpha}(V)$. Since $P_{\alpha}$ and $\pi_{\alpha}$
are surjective and continous maps, $(P_{\alpha}  \circ \pi_{\alpha})^{-1}(V)$ is an open nonempty set. Then, $(\pi^{'}_{\alpha}\circ P_{\infty})^{-1}(V)= (P_{\infty})^{-1} (U)$ is an open nonempty set,
which is impossible for $U\cap P_{\infty}(\sigma (\mathcal L,  (I_{\alpha})_{\alpha\in I}))$ is the empty set.
\end{proof}   
 
\indent We now state our main result.\par

\begin{thm} Let $X$ be a complex Banach space and 
$\sigma (.)$ a joint spectrum for complex solvable finite dimensional 
Lie algebras of operators defined on $X$. Then, for each complex quasi-solvable 
Lie subalgebra $\mathcal L$ of ${\mathcal L}(X)$, 
there is a uniquely well defined map, also denoted by $\sigma(.)$, such that the following conditions are fullfilled.

{\rm (i)} $\sigma(\mathcal L)$ is a subset of characters of $\mathcal L$ and a 
compact nonempty subset of $\Pi_{\alpha\in \mathbb I}
\sigma (I_{\alpha})$, where $\mathbb I$ denotes the set of all
complex solvable finite dimensional ideals of $\mathcal L$, 

{\rm (ii)} if $\mathcal H$ is a complex solvable finite dimensional ideal of $\mathcal L$,
then $\sigma (\mathcal H)$ coincides with the joint spectrum of $\mathcal H$ 
defined in the finite dimensional case,

{\rm (iii)} if $\mathcal M$ is a subalgebra of $\mathcal L$, and $\mathcal H$ is a Lie
ideal of $\mathcal M$, then for the joint spectrum the projection property for ideals holds., i.e., 
$$
\sigma (\mathcal M)\mid \mathcal H=\sigma (\mathcal H).
$$    

\end{thm}

\begin{proof}
\indent By the propositions we have proved, in order to see the existence of such joint spectrum, it is enough to consider the set $\mathbb I$ of all complex solvable finite dimensional  ideals of $\mathcal L$, and the presentation of $\mathcal L=\sum_{\alpha\in\mathbb I}I_{\alpha}$. In fact, if $\mathcal M$ is a Lie subalgebra, or ideal, of $\mathcal L$, an easy calculation shows that $\mathcal M=\sum_{\alpha\in \mathbb I}\mathcal M\cap I_{\alpha}$, thus, we may define $\sigma (\mathcal M)$ as 
$$
\sigma (\mathcal M)=\sigma (\mathcal M , (\mathcal M\cap I_{\alpha})_{\alpha\in \mathbb I}).
$$
With this definition, the joint spectrum satisfies properties (i)-(iii). \par

\indent In order to prove that this map is uniquely determined, we proceed as follows.\par

\indent Let us suppose that $\tilde\sigma$ is an assigment which satisfies the previous conditions. Then, by Proposition 3.3, and by conditions (i)-(iii), $\tilde\sigma(\mathcal L)\subseteq \sigma (\mathcal L)$.\par

\indent On the other hand, if $f\in \sigma(\mathcal L)\setminus \tilde\sigma(\mathcal L)$, 
since both joint spectra are compact subsets of $\Pi_{\alpha\in \mathbb I}\sigma (I_{\alpha})$, there is an open set U, which may be chosen in the base of the topology of $\sigma (\mathcal L)$, such that
$f\in U$ and $U\cap \tilde\sigma(\mathcal L)$ is the empty set. Moreover, by [4, Chapter VIII, Section 3, Lemma 3.12] there is $\alpha\in\mathbb I$ and $V$ in the topology of $\sigma (I_{\alpha})$ such that $U=\pi_{\alpha}^{-1}(V)$.\par

\indent Now, by condition (ii) we have that $f\mid 
I_{\alpha}\in \sigma (I_{\alpha} )=\tilde\sigma (\mathcal L)\mid I_{\alpha}$. Thus, by condition (iii), there is $g\in\tilde\sigma (\mathcal L)$ such that $g\mid I_{\alpha}=f\mid I_{\alpha}$. However, since $g\mid I_{\alpha}=f\mid I_{\alpha}$, $g$ belongs to $U$, which is impossible for $g\in\tilde\sigma (\mathcal L)$.
\end{proof}
\indent As we have pointed out, this construction may be applied to the
Taylor and the S\l odkowski joint spectra for complex solvable finite 
dimensional Lie algebras of operators, see [2], [3] and [6], 
and then we extend 
these joint spectra to quasi solvable Lie algebras of operators.
Moreover, this construction gives a non commutative version
of the one developed by Z. S\l odkowski and W. Zelazko in [9],
which, in general, differs from the one considered by D. Beltita
in [1]. In fact, in the solvable finite dimensional case the joint spectrum of [1] does not, in general, coincide
with the one of [2], [3] and [6].\par

\bibliographystyle{amsplain}

\vskip.5cm
Enrico Boasso\par
\noindent E-mail address: enrico\_odisseo@yahoo.it

\end{document}